\documentclass[a4paper,12pt]{article}
\usepackage[english]{babel} 
\usepackage{amsfonts,amssymb,amsmath,a4wide,amsbsy} 
\usepackage{amsthm} 

\usepackage{graphicx}

\newcommand{\R}{\ensuremath{\mathbb{R}}}

\renewcommand{\P}{{\mathbb{P}}}

\newcommand{\thefont}[2]{\fontsize{#1}{#2}\fontshape{n}\selectfont}
\def\ind{\rlap{\thefont{10pt}{12pt}1}\kern.16em\rlap{\thefont{11pt}{13.2pt}1}\kern.4em}

\providecommand{\be}{\begin{equation}}
\providecommand{\ee}{\end{equation}}

\theoremstyle{plain}
\newtheorem{thm}{Theorem}
\newtheorem{lem}[thm]{Lemma}
\newtheorem{cor}[thm]{Corollary}

\theoremstyle{definition}

\newtheorem{rem}[thm]{Remark}

\title{Sharp upper bounds for the deviations from the mean of the sum of independent Rademacher random 
variables}
\author{Harrie Hendriks and Martien C.A. van Zuijlen}
\date{}

\begin{document}

\maketitle
\centerline{IMAPP, MATHEMATICS} 
\medskip
\centerline{RADBOUD UNIVERSITY  NIJMEGEN}
\begin{center}
	{
	Heyendaalseweg 135\\ 
	6525 AJ Nijmegen\\ 
	The Netherlands\\
	e-mail authors: H.Hendriks@math.ru.nl, M.vanZuijlen@math.ru.nl
}\end{center}

\begin{abstract}

  For a fixed unit vector $a=(a_1,a_2,...,a_n)\in S^{n-1}$, i.e. $\sum_{i=1}^na_i^2=1$, we consider the $2^n$ sign 
 vectors $\epsilon=(\epsilon_1,\epsilon_2,...,\epsilon_n)\in \{-1,1\}^n$ and the 
 corresponding scalar products $a.\epsilon=\sum_{i=1}^n a_i\epsilon_i.$ 
 In \cite{HK} the following old conjecture has been reformulated. It states  
 that among the $2^n$ sums of the form $\sum\pm a_i$  there are not more
   with $|\sum_{i=1}^n \pm a_i|>1$ than there  are with 
  $|\sum_{i=1}^n \pm a_i|\leq 1.$ The result is of interest in itself, but has also 
   an appealing reformulation in probability theory and in geometry.
 In this paper we will solve an extension of this problem in the uniform case where all the $a$'$s$ are equal. More precisely, for $S_n$ being a sum of $n$ independent Rademacher random variables,  we will give, for several values of $\xi,$ precise lower bounds
for the probabilities $$P_n:=\P\{-\xi\sqrt{n}\leq S_n \leq \xi\sqrt{n}\}$$ 
or equivalently for $$Q_n:=\P\{-\xi\leq T_n \leq \xi\},$$
where $T_n$ is a standardized Binomial random variable with parameters $n$ and $p=1/2.$
 These lower bounds are sharp and much better than for instance the bound that can be obtained  from application of the Chebishev inequality. In case $\xi=1$  Van Zuijlen solved this problem in \cite{MVZ}.  We remark that our bound will have nice applications in probability theory and especially in random walk theory. 
\end{abstract}

\section{Introduction and result}
	Let  $\epsilon_1,\epsilon_2,...,$ be a sequence of i.i.d. Rademacher 
random variables and for positive integers $n$ let $a_n=(a_{1n},a_{2n},...,a_{nn})$ be a unit-vectors in $\R^n,$ so that
$\sum_{i=1}^na_{in}^2=1$. The following problem has been presented in \cite{G} 
and is attributed to B. Tomaszewski. In \cite{HK}, Conjecture 1.1, this old 
problem has been reformulated as follows:
  $$\P(|a_{1n}\epsilon_1+a_{1n}\epsilon_2+...+a_{1n}\epsilon_n|\leq 1)\geq \frac{1}{2}, \;\;\hbox{for}\;\;n=1,2,...$$ 
  This conjecture is at least  25 years old and seems still
     to be unsolved. In the uniform case where,
    $$a_{1n}=a_{2n}=...=a_{nn}=n^{-1/2},$$ 
  the maximum possible value of $\frac{S_n}{\sqrt n}$ is 
$\sqrt n,$ where 
\begin{equation}\label{vgl0}
 S_n:=\epsilon_1+\epsilon_2+...+\epsilon_n
\end{equation}
 and the conjecture, stating that for integers $n\geq 2,$
$$ \P\{|S_n|\leq  \sqrt{n}\}=\P\{|\sum_{i=1}^n\epsilon_i|\leq  \sqrt n\}\geq 1/2,$$ 
  has been solved recently by M.C.A. van Zuijlen. See \cite{MVZ}.
 It means that  at least $50\%$ of the probability mass is between minus one and one standard deviation from the mean, which is quite remarkable. We note that 
 \begin{itemize}
 \item[i)] $S_n$ can be easily expressed in terms of sums of independent Bernoulli(1/2) random variables since $(\epsilon_i+1)/2$ are independent Bernoulli random variables and hence $S_n$ is distributed as $2B_n-n$, where $B_n$ is a binomial random variable with parameters $n$ and $1/2.$ It follows that $S_n/\sqrt{n}$ is distributed as $T_n$, where $T_n$ is a binomial random variable with parameters $n$ and $p=1/2.$
 \item[ii)] easy calculations show that the  sequence  $(P_n)$ is 
not monotone in $n$.
\end{itemize}
  In this paper we shall generalize Van Zuijlen's result and derive sharp lower bound 
  for probabilities concerning $\xi$ standard deviations:
  \begin{equation}\label{vgl1}
P_n:=\P\{|S_n|\leq \xi\sqrt n\}= 
\P\{|\sum_{i=1}^n\epsilon_i|\leq \xi\sqrt n\},
\end{equation} where $\xi\in (0,1].$
%
Note that trivially $$P_1=\begin{cases} 1,&\;\;\text{for}\;\; \xi =1;\\
0,&\;\;\text{for}\;\; \xi<1. \\
\end{cases} $$
Throughout the paper $n$ and $k$ will denote nonnegative integers. 
Our result is as follows.

\begin{thm}\label{the thm}
 Let $ \epsilon_1,\epsilon_2,...,\epsilon_n$ be independent Rademacher random variables, so that $$\P\{\epsilon_1=1\}=\P\{\epsilon_1=-1\}=1/2$$ and let
  $S_n$  and 
 $P_n$ be defined as in (\ref{vgl0}) and (\ref{vgl1}), where 
 $\xi\in (0,1].$
Define
$$n_k=2\left\lceil\frac{\frac{k^2}{\xi^2}+k}2\right\rceil-k-1,\quad C_k=\{n: n_k\leq n < n_{k+1}\},
\hbox{ and }Q_k^-:=P_{n_{k+1}-1}.$$
Then,  with $\Phi$ indicating the standard normal distribution function, we have for $k\ge0$
  \begin{itemize} 
  \item [a.] $P_n=\P\{|S_n|\le \xi\sqrt n\}=\P\{|S_n|\le k\}$, for $n\in C_k$,
  \item [b.]  $Q_k^-=\min_{n\in C_k}P_n,$
  \item [c.]  the sequence $(Q_k^-)$ is strictly monotone increasing in $k$,
  \end{itemize}
  Moreover,
  \begin{itemize} 
  \item [d.]  $\lim_{k\rightarrow \infty} Q_k^-  =1-2 \Phi(\xi)$,
  \item [e.]  $  Q_1^-=P_{n_2-1}\leq P_n$, for all $n\ge n_1$.
  \end{itemize}
  \end{thm}
\noindent
A consequence of Theorem \ref{the thm} is the following result.
\begin{cor}\label{Cor2}
For $n\geq 2$ we have
$$P_n\geq 
\begin{cases} 1/2,&\;\;\text{for}\;\; \xi =1;\\
3/8,&\;\;\text{for}\;\; \xi \in [\sqrt{2/3},1); \\
5/16,&\;\;\text{for}\;\; \xi \in [\sqrt{1/2},\sqrt{2/3}) .\\
\end{cases} $$ 
More generally, if $0<\xi\le1$ and $n_2\ge3$, $n_2$ odd,
then we have
$$\frac{2}{\sqrt{n_2+1}}\leq \xi<\frac{2}{\sqrt{n_2-1}},\qquad
n_1=2\left\lceil\frac{n_2-3}{8}\right\rceil
$$ 
and for all $n\geq n_1$
 $$P_n\geq P_{n_2-1}=\binom{n_2-1}{(n_2-1)/2}2^{-(n_2-1)}.$$
 \end{cor}

\noindent
It is worthwhile to clarify in a plot the structure of the probabilities
$P_n(\ell)=\P\{|S_n|=\ell\}$, where $n$ and $\ell$ are nonnegative integers such that
$n+\ell$ is even. See Figure \ref{Fig1}.

\begin{figure}
\label{Fig1}
\includegraphics{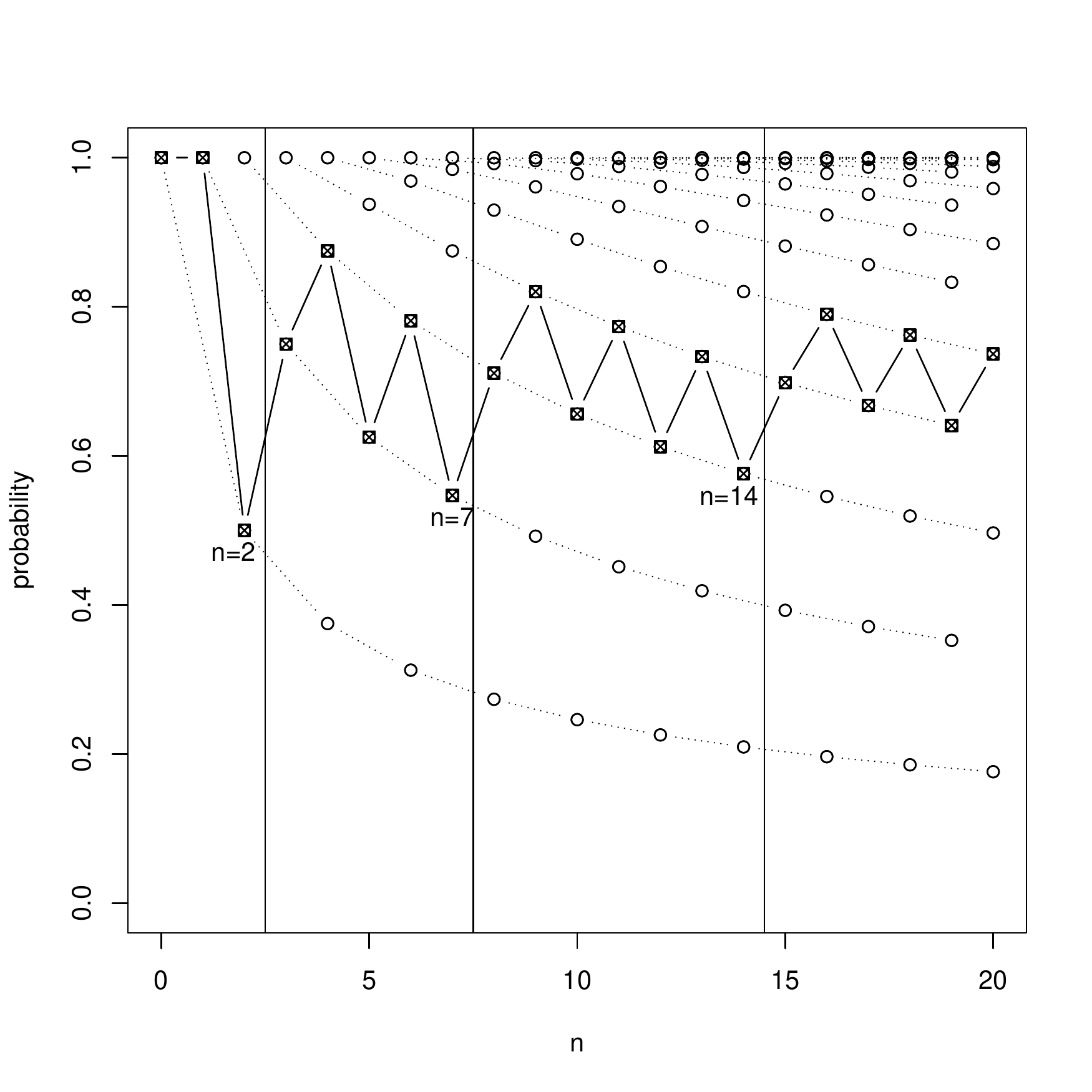}
\caption{Graph of probabilities $P_n(\ell)$, $n+\ell$ even. 
(Dotted lines connect points with
constant $\ell=0,1,2,\ldots$, upwards in graph. The square symbols indicate the points
$(n,P_n)$ for $\xi=1$. The vertical lines separate the regions $C_k$, $k=1,2,\ldots$)}
\end{figure}

\section{Preliminaries}
Be given independent Rademacher random variables $\varepsilon_i$, $i=1,2,3,\ldots$,
as defined in Theorem \ref{the thm}, and let $S_n=\sum_{i=1}^n\varepsilon_i$ such that $S_0=0$.
Define
$$P_n(k)=\P(|S_n|\le k).$$
Since $n+S_n$ is even it follows that $P_n(k)=P_n(k-1)$ if $n+k$ is odd. 
\\
A basic property is the symmetry of the distribution of $S_n$:
$$\P\{S_n=k\}=\P\{S_n=-k\}.$$
Moreover, $\varepsilon_n$ being independent of $S_{n-1}$,
\begin{eqnarray*}
\lefteqn{\P\{(S_{n-1}=k+1\,\&\,\varepsilon_n=-1)\hbox{ or }(S_{n-1}=-k-1\,\&\,\varepsilon_n=+1)\}}
\qquad\qquad\qquad\qquad\qquad
\\
&=&
2\P\{S_{n-1}=k+1\,\&\,\varepsilon_n=-1\}
=
\P\{S_{n-1}=k+1\}
\end{eqnarray*}
and, replacing $\varepsilon_n$ by the equally distributed $-\varepsilon_n$,
$$
\P\{(S_{n-1}=k+1\,\&\,\varepsilon_n=+1)\hbox{ or }(S_{n-1}=-k-1\,\&\,\varepsilon_n=-1)\}
=
\P\{S_{n-1}=k+1\}.
$$
This leads to the following properties for $P_n(k)$.
\begin{rem}\label{remark}\emph{
Suppose $n+k$ is even, $n\ge1$, 
then
\begin{eqnarray*}
P_{n}(k)&=&P_{n-1}(k-1)+\P\{(S_{n-1}=k+1\,\&\,\varepsilon_{n}=-1)\hbox{ \rm or }(S_{n-1}=-k-1\,\&\,\varepsilon_n=+1)\}
\\&=&P_{n-1}(k-1)+\P\{S_{n-1}=k+1),
\\
P_n(k)&=&P_{n-1}(k+1)-\P\{(S_{n-1}=k+1\,\&\,\varepsilon_{n}=+1)\hbox{ \rm or }(S_{n-1}=-k-1\,\&\,\varepsilon_n=-1)\}
\\&=&P_{n-1}(k+1)-\P\{S_{n-1}=k+1).
\end{eqnarray*}
Suppose $n+k$ is even, $n\ge1$, then
\begin{eqnarray*}
\P\{S_{n-1}=k-1\}&=&\binom{n-1}{\frac{n+k}2-1}2^{-(n-1)}=\frac{n+k}n\binom{n}{\frac{n+k}2}2^{-n}=\frac{n+k}n\P\{S_n=k\},
\\
\P\{S_{n-1}=k+1\}&=&\P\{S_{n-1}=-k-1\}=\frac{n-k}n\P\{S_n=-k\}=\frac{n-k}n\P\{S_n=k\}.
\end{eqnarray*}
Suppose $n+k$ is even, $n\ge1$, $k-1\ge0$. 
It follows that
\begin{eqnarray*}
P_{n-1}(k-1)-P_{n+1}(k-1)&=&P_n(k-2)+\P\{S_{n-1}=k-1\}-P_n(k-2)-\P\{S_{n}=k\}
\\
&=&
\P\{S_{n-1}=k-1\}-\P\{S_{n}=k\}=\frac kn\P\{S_{n}=k\}.
\end{eqnarray*}
Furthermore, for $n\ge k\ge0$,
\begin{eqnarray*}
\frac{\P\{S_{n}=k\}}{\P\{S_{n+2}=k\}}
&=&
\frac{\P\{S_{n}=k\}}{\P\{S_{n+1}=k+1\}}\times\frac{\P\{S_{n+1}=k+1\}}{\P\{S_{n+2}=k\}}=
\frac{n+1+k+1}{n+1}\times\frac{n+2-k}{n+2}
\\&=&
\frac{(n+2)^2-k^2}{(n+2)^2-(n+2)}.
\end{eqnarray*}
In particular, if $k^2\le n+2$ then $\P\{S_{n}=k\}\ge\P\{S_{n+2}=k\}$, with equality only if $k^2=n+2$.
}\end{rem}
\begin{cor}
Suppose $n+k$ is even, $n>k$ (i.e. $n\ge k+2$), then $P_n(k)=P_n(k+1)<P_m(k+1)$ for all $m<n$.
\end{cor}
\begin{proof}
According to the above Remark \ref{remark}
$$P_n(k)<P_{n-2}(k)<\cdots<P_{k+2}(k)<P_k(k)=1,$$
$$P_n(k)<P_{n-1}(k+1)<P_{n-3}(k+1)<\cdots<P_{k+1}(k+1)=1.$$
\end{proof}
\begin{thm}\label{mainthm}
Suppose $k\ge1$, $n\ge k$ and $n+k$ is even.
If $n+2\ge k^2$ and $\ell\ge0$, such that $n+1+2\ell<\frac{(k+1)^2}{k^2}(n+2)$, then
$P_n(k-2)<P_{n+1+2\ell}(k-1)$.
\end{thm}
\begin{proof}
For $\ell=0$ we remark that
$P_{n+1}(k-1)-P_n(k-2)=\P\{S_n=k\}>0$.
For $\ell\ge1$ it is sufficient to show
$P_{n+1}(k-1)-P_n(k-2)>P_{n+1}(k-1)-P_{n+1+2\ell}(k-1)$.
Since
\begin{eqnarray*}
P_{n+1}(k-1)-P_n(k-2)&=&\P\{S_n=k\}\quad\hbox{and}
\\
P_{n+1}(k-1)-P_{n+1+2\ell}(k-1)&=&\sum_{i=1}^\ell\frac{k}{n+2i}\P\{S_{n+2i}=k\},
\end{eqnarray*}
this inequality will follow from the claim
$\P\{S_n=k\}>\sum_{i=1}^\ell\frac{k}{n+2i}\P\{S_{n+2i}=k\}$.
\\
If $\ell=1$, then $\P\{S_n=k\}\ge\P\{S_{n+2}=k\}>0$
and $1>\frac k{k+2}\ge \frac k{n+2}$.
If $\ell>1$ or $n+2>k^2$, then  $\P\{S_n=k\}>\P\{S_{n+2\ell}=k\}>0$, so that
it is sufficient to show that $\sum_{i=1}^\ell\frac{k}{n+2i}\le1$.
Theorem \ref{mainthm} now follows from Lemma \ref{main lemma} in the Appendix.
\end{proof}
\begin{cor}\label{generalcor}
Let $n_k$, $k=1,2,3,\ldots$, be an increasing sequence of integers such that $n_1\ge0$, $n_k+k$ is odd,
$n_k+1\ge k^2$ and $n_{k+1}-1<\frac{(k+1)^2}{k^2}(n_k+1)$.
Then for $n_k\le m<n_{k+1}-1$ we have
$$\P\{S_{n_2-1}=0\}=P_{n_2-1}(0)<\cdots<P_{n_k-1}(k-2)<P_{n_{k+1}-1}(k-1)<P_m(k)
,$$
which for $k=1$ reduces to
$$
\P\{S_{n_2-1}=0\}=P_{n_2-1}(0)<P_m(1).$$
\end{cor}
\begin{proof}
For $k\ge2$ or $n_1\ge2$, apply Theorem \ref{mainthm} with $n=n_k-1$ and $\ell=(n_{k+1}-n_k-1)/2$.
In case $k=1$ and $n_1=0$ we have $n_3=3$ and the claim in the corollary is trivial.
\end{proof}

\section{The original context}\label{original context}
Let $\xi>0$ and consider the event $\{|S_n|\le\xi\sqrt n\}$.
Let $k$ be the integer such that $n+k$ is even and $k\le\xi\sqrt n<k+2$.
Then $\{|S_n|\le\xi\sqrt n\}=\{|S_n|\le k\}$.
Notice that such $k$ satisfies the inequalities
$$\frac{n+k}2\le\frac{n+\xi\sqrt n}2<\frac{n+k+2}2=\frac{n+k}2+1$$
so that 
$\frac{n+k}2=\left\lfloor\frac{n+\xi\sqrt n}2\right\rfloor$ and hence
$$k=\kappa(n):=2\left\lfloor\frac{n+\xi\sqrt n}2\right\rfloor-n.$$
It follows immediately that $\kappa(n+2)\ge\kappa(n)$, $\kappa(0)=0$.
Moreover
\begin{eqnarray*}
\kappa(n+1)-\kappa(n)
&=&
2\left\lfloor\frac{n+1+\xi\sqrt {n+1}}2\right\rfloor-n-1-2\left\lfloor\frac{n+\xi\sqrt n}2\right\rfloor+n
\\&=&
2\left\lfloor\frac{n+1+\xi\sqrt {n+1}}2\right\rfloor-2\left\lfloor\frac{n+\xi\sqrt n}2\right\rfloor-1,
\end{eqnarray*}
so that $\kappa(n+1)-\kappa(n)$ is odd and greater than
or equal to $-1$.
\\
It is interesting to notice the following fact:
If $a$ and $b$ are nonnegative integers we have
$
\xi\sqrt{a}<\kappa(a)+2\hbox{ and } \kappa(b)\le\xi\sqrt{b},
$
so that 
\begin{equation}\label{distance}
a<\frac{(\kappa(a)+2)^2}{\kappa(b)^2}b.
\end{equation}
From the inequality $\lfloor a \rfloor -\lfloor b \rfloor <a-b+1$ one concludes
\begin{eqnarray*}
\kappa(n+1)-\kappa(n)
&<&
n+1+\xi\sqrt {n+1}-n-\xi\sqrt n+1
\\&=&
2+\xi\sqrt {n+1}-\xi\sqrt n
=2+\frac\xi{\sqrt {n+1}+\sqrt n}
.
\end{eqnarray*}
It follows that for $\xi\le1$, $\kappa(n+1)-\kappa(n)\le1$, since then it is an odd number strictly less than 3.
As a matter of fact, already for $\xi<2\sqrt2$ we have $\kappa(n+1)-\kappa(n)\le1$.
In the sequel assume that $\xi\le1$. Then we have the basic properties
$$\kappa(n+1)-\kappa(n)=\pm1,\quad \kappa(0)=0,\quad \kappa(n)\le\xi\sqrt{n},\quad \kappa(n+2)\ge\kappa(n).$$
\\
For $k\ge1$, define 
\begin{equation}\label{nk}
n_k:=\min\{n\mid \kappa(n+1)\ge k\}=2\left\lceil\frac{\frac{k^2}{\xi^2}+k}2\right\rceil-k-1.
\end{equation}
It is clear that $n_k$ is strictly increasing in $k$.
Moreover $\kappa(n_k+1)=k$ and $\kappa(n_k)=k-1$ and
$\kappa(n_k-1)=k-2$ (if $k\ge2$ or if $k=1$ and $n_1\ge1$).
Also $n_k+k$ is odd.
In case $\xi=1$ it is easy to see that $n_k=k^2-1$.
Since $n_k$ is decreasing in $\xi$, it follows for $\xi\le1$ that $n_k\ge k^2-1$.
\\
Notice that for $m\le n_{k+1}$ we have $\kappa(m)<k+1$, so that $\kappa(m)\le k$.
On the other hand, if $\kappa(m)\le k-2$, it follows for all $n\le m$ that $\kappa(n)\le k-1$, so that $m< n_k+1$
and since $\kappa(n_k)=k-1$ it follows that $m<n_k$. 
We conclude that for
$n_k\le m\le n_{k+1}$ we have $k-1\le\kappa(m)\le k$, so that
\begin{equation}\label{Pmk}
P_m(k)=\P\{|S_m|\le \xi\sqrt{m}\}.
\end{equation}
\\
Since $k\le\xi\sqrt{n_k+1}$ and $0<\xi\le1$ we have $k^2\le n_k+1$.
\\
From Inequality (\ref{distance}) we obtain
$$n_{k+1}-1<\frac{(k+1)^2}{k^2}(n_k+1).$$
Provided that $n_k-1\ge k$, Theorem \ref{mainthm} leads to the inequality $P_{n_k-1}(k-2)<P_{n_{k+1}-1}(k-1)$.
Notice that $n_k\ge2$ if $k\ge2$ or if $k=1$ and $\xi<1$.
The main result, Theorem \ref{the thm}, in fact follows 
from Corollary \ref{generalcor}.
More specifically,
\begin{cor}\label{maincor}
Let $\xi\le1$ and $n_k$, $k=1,2,3,\ldots$ be defined as (\ref{nk}).
Then, for $k\ge2$ and all $m$ satisfying  $n_{k-1}\le m<n_k$, we have
$$\P\{S_{n_2-1}=0\}=P_{n_2-1}(0)\le P_{n_k-1}(k-2)<P_m(k-1)=\P\{|S_m|\le\xi\sqrt{m}\}.$$
In particular, for $m\ge n_1$ we have
$\P\{S_{n_2-1}=0\}\le\P\{|S_m|\le\xi\sqrt m\}$, with equality only for $m=n_2-1$.
\end{cor}
\begin{proof}[Proof of Theorem \ref{the thm}]
Claim a) has been dealt with in (\ref{Pmk}).
Claims b), c) and e) follow directly from the above Corollary \ref{maincor}.
Finally, Claim d) follows from the Central Limit Theorem.
\end{proof}
It is the condition $\xi\le1$ that implies that $n_k+1\ge k^2$, needed in Corollary \ref{generalcor}.
For $\xi>1$ it is no longer true that $P_{n_{k+1}-1}(k-1)>P_{n_k-1}(k-2)$ as can be seen from the following examples.
For $\xi=\sqrt2$, we have $n_4=7,n_5=12$ and $P_{n_5-1}(3)=\frac{99}{128}<\frac{100}{128}=P_{n_4-1}(2)$.
For $\xi=1.1$ and $k=22$:
$n_{22}=399=20^2-1;n_{23}=438$ and $P_{{n_{23}}-1}(21) <0.70745< P_{n_{22}-1}(20)$.
For $\xi=1.01$ and $k=202$, $n_k=39999=200^2-1$, $n_{k+1}=40398$ and $P_{{n_{203}}-1}(201) <0.6851152< P_{n_{202}-1}(200)$.
\\
\\
Concerning Corollary \ref{Cor2} we note the following.
It is straightforward to see that
$$ n_2=\begin{cases} 3,&\;\;\text{for}\;\; \xi =1,\\
5,&\;\;\text{for}\;\; \xi \in [\sqrt{2/3},1),\\
7,&\;\;\text{for}\;\; \xi \in [\sqrt{1/2},\sqrt{2/3}),\\
\end{cases} $$
so that
$$\P\{-1\leq S_{n_2-1}\leq 1\}=
\begin{cases} 1/2,&\;\;\text{for}\;\; \xi =1,\\
3/8,&\;\;\text{for}\;\; \xi \in [\sqrt{2/3},1),\\
5/16,&\;\;\text{for}\;\; \xi \in [\sqrt{1/2},\sqrt{2/3}).\\
\end{cases} $$
More generally, from definition (\ref{nk}) we have
$n_2=2\left\lceil\frac{\frac4{\xi^2}+2}2\right\rceil-2-1=2\left\lceil\frac2{\xi^2}\right\rceil-1$,
which is equivalent to
$$
\frac2{\sqrt{n_2+1}}\le\xi<\frac2{\sqrt{n_2-1}}
$$
and to
$$
\frac{n_2+3}8<\frac{\frac1{\xi^2}+1}2\le\frac{n_2+5}8.
$$
Since $n_2$ is odd, the open interval 
$(\frac{n_2+3}8,\frac{n_2+5}8)$ does not contain an integer.
Thus, for such $\xi$,
$$n_1=2\left\lceil\frac{\frac1{\xi^2}+1}2\right\rceil-1-1=
2\left\lceil\frac{n_2+5}8\right\rceil-2=2\left\lceil\frac{n_2-3}8\right\rceil
$$
and for all $n\geq n_1$,
 we have from Theorem \ref{the thm}
 $$P_n\geq P_{n_2-1}=\P\{S_{n_2-1}=0\}=\binom{n_2-1}{(n_2-1)/2}2^{-(n_2-1)}.$$

\section{Examples}\label{Examples}

In case $\xi=\sqrt{1/2},$  we obtain
 for   $k\in\{1,2,...\}$
 $$n_k=2\left\lceil\frac{\frac{k^2}{\xi^2}+k}2\right\rceil-k-1=
2\left\lceil\frac{2k^2+k}2\right\rceil-k-1= \begin{cases}2k^2-1,\;\; \text{for k = even} ,\\
2k^2,\;\;\;\;\;\;\;\; \text{for k = odd}.
\end{cases}$$
 In this case  $n_1=2,n_2=7,n_3=18,n_4=31,$ so that $C_1=[2,6],C_2=[7,17],C_3=[18,30]$  and the minimal value in $C_1$ is  
 $$P_{n_2-1}=P_{6}=P\{-\xi\sqrt{n_2-1}\leq S_{n_2-1}\leq \xi\sqrt{n_2-1}\}= 
P\{-1\leq S_{6}\leq 1\}=P\{S_{6}=0\}= $$
$$=P\{B_{6}=3\}=\frac{5}{16}.$$
Here the $B_n$ denote the binomial random variables as in the Introduction. Also,
$$
P_{n_3-1}=P_{17}=\P\{-2\le S_{17}\le2\}=2\P\{S_{17}=1\}=2\P\{B_{17}=9\}=\frac{12155}{32768}\ge
\frac{10240}{32768}=\frac5{16}.
$$

\noindent In case $\xi=\sqrt{2/3},$ hence we obtain
 for   $k\in\{1,2,...\}$

$$n_k=2\left\lceil\frac{\frac{3k^2}{2}+k}2\right\rceil-k-1
=2\left\lceil\frac{3k^2+2k}4\right\rceil-k-1=
\begin{cases}\frac32k^2-1,\;\; \text{for k = even} ,\\
\frac32k^2+\frac12,\;\; \text{for k = odd}.
\end{cases}.   $$
Hence, $n_1=2 ,n_2=5, n_3=14, n_4=23, n_5=38, n_6=53$  with blocks $C_1=[2,4],C_2=[5,13], C_3=[14,22], C_4=[23,37],C_4=[38,52].$
The minimal value in $C_1$ is obtained for 
$$P_{n_2-1}=P\{-\xi\sqrt{n_2-1}\leq S_{n_2-1}\leq \xi\sqrt{n_2-1}\}= 
P\{-1\leq S_{4}\leq 1\}=P\{S_{4}=0\}=$$
$$=P\{B_{4}=2\}=\frac{3}{8}.$$
Also,
$$
P_{n_3-1}=\P\{-2\le S_{13}\le2\}=2\P\{S_{13}=1\}=2\P\{B_{13}=7\}=\frac{429}{1024}\ge
\frac{384}{1024}=\frac3{8}.
$$

\noindent In case $\xi=1$  we obtain for  $k\in\{1,2,...\}$
 $$n_k=2\left\lceil\frac{\frac{k^2}{\xi^2}+k}2\right\rceil-k-1=k^2-1.$$ 
We obtain for integers $k\geq 2,$  $C_k=\{k^2-1,k^2,...,(k+1)^2-2\},$ with length $m_k=2k+1.$
Now $n_1=0,n_2=3,n_3=8,n_4=15,$ so that $C_1=[0,2],C_2=[3,7],C_3=[8,14].$
  The minimal value in $C_1$ is obtained for  
  $$P_{n_2-1}=P\{-\xi\sqrt{n_2-1}\leq S_{n_2-1}\leq \xi\sqrt{n_2-1}\}=P\{-1\leq S_{2}\leq 1\}=P\{S_{2}=0\}=$$
$$=P\{B_{2}=1\}=\frac{1}{2}.$$
The minimal value in $C_2$ is obtained for $n= n_3-1=7$ and equals
 $$P_{n_3-1}=P_7=\P\{-2\leq S_7\leq 2\}=2\P\{S_7=1\}=2\P\{B_7=4\}
 =\frac{35}{64}\ge\frac{32}{64}=\frac12.$$

\section*{Appendix}
In this section we state and prove the lemma needed in the proof of Theorem \ref{mainthm}.
\begin{lem}\label{main lemma}
Suppose $k\ge1$ and $n+k$ even.
If $n+2\ge k^2$ and $\ell\ge0$, such that $n+1+2\ell<\frac{(k+1)^2}{k^2}(n+2)$, then
$\sum_{i=1}^\ell\frac{k}{n+2i}\le1$.
\end{lem}
\begin{proof}
The goal is to prove the inequality
$$
\sum_{i=1}^\ell\frac{k}{n+2i}\le1.
$$
It is easy to see that $k/(n+2i)+k/(n+2\ell+2-2i)$ is decreasing in $i$ for $i\le\ell/2$.
Therefore it is sufficient to prove $k/(n+2)+k/(n+2\ell)<2/\ell$, or equivalently
\begin{equation}\label{simple ineq}
\frac{\ell}{n+2}+\frac{\ell}{n+2\ell}\le\frac2k.
\end{equation}
Since the left hand is increasing in $\ell$, it is sufficient for given $n$ to consider the maximally allowed $\ell$.
In the same way, given $\ell$ it is sufficient to prove the inequality for the minimally allowed $n$.
\\
The condition $n+1+2\ell<\frac{(k+1)^2}{k^2}(n+2)$ is equivalent to $2\ell-1<\frac{2k+1}{k^2}(n+2)$.
Thus for any $n$ such that $n+2\ge k^2$, $\ell=k$ is an allowed value for $\ell$. The corresponding minimal value of
$n$ is $n=k^2-2$.
It follows that for $\ell\le k$ Inequality (\ref{simple ineq}) holds:
$$\frac2k-\frac \ell{n+2}-\frac \ell{n+2\ell}\ge\frac2k-\frac k{n+2}-\frac k{n+2k}\ge\frac2k-\frac k{k^2}-\frac k{k^2}=0.$$
Next consider the case $\ell\ge k+1$.
Then the condition $2\ell-1<\frac{2k+1}{k^2}(n+2)$ leads to
\begin{equation}
\label{ineq n+2}
n+2>\frac{2\ell-1}{2k+1}k^2=k^2+(\ell-k-1)(k-\frac12)+\frac{\ell-k-1}{2(2k+1)}
\ge k^2+(\ell-k-1)(k-\frac12).
\end{equation}
In case $\ell=k+1$ it means that $n+2> k^2$, and because $n+k$ is even, $n+2\ge k^2+2$.
Substituting $\ell=k+1$ and $n=k^2$ we get Inequality (\ref{simple ineq}) for $\ell=k+1$:
$$\frac2k-\frac\ell{n+2}-\frac\ell{n+2\ell}
=
\frac{2(k^2+2k+4)}{k(n+2)(n+2\ell)}\ge0.
$$
For the case $\ell\ge k+2$ we conclude from Inequality (\ref{ineq n+2}) that $n+2\ge k^2+(\ell-k-1)(k-\frac12)+\frac12$.
Substituting $\ell=k+2+j$ and $n=k^2+(\ell-k-1)(k-\frac12)-\frac32$ we get
$$\frac2k-\frac\ell{n+2}-\frac\ell{n+2\ell}
=
\frac{2j(k^2-2)+j^2(2k-3)}{k(n+2)(n+2\ell)}.
$$
Since the right hand side is nonnegative for $j\ge0$ and $k\ge2$ we established Inequality (\ref{simple ineq})
for $k\ge2$ and $\ell\ge k+2$.
\\
If $k=1$, $n$ odd, then from $n+1+2\ell<\frac{(k+1)^2}{k^2}(n+2)$ it follows that
$2\ell-1<3(n+2)$, which implies  $2\ell-1\le 3(n+2)-2$, so that the maximal $\ell$ is $\ell=(3n+5)/2$.
Again 
$$\frac2k-\frac{\ell}{n+2}-\frac{\ell}{n+2\ell}=\frac{(n+1)(n+5)}{2(n+2)(n+2\ell)}\ge0.$$

\end{proof}

\end{document}